\documentclass[11pt]{article}

\usepackage[utf8]{inputenc}
\usepackage[T1]{fontenc}
\usepackage{lmodern}
\usepackage[a4paper,margin=1in]{geometry}
\usepackage{amsmath,amssymb,amsthm,mathtools, amsfonts}
\usepackage{microtype}
\usepackage{hyperref}
\usepackage{tabularx,booktabs}
\usepackage{footmisc}
\DeclareMathOperator{\Log}{Log}

\newcommand{\firstpagefoot}[1]{%
  \begingroup
  \renewcommand\thefootnote{}\footnote{#1}%
  \addtocounter{footnote}{-1}%
  \endgroup
}
\newtheorem{theorem}{Theorem}[section]

\theoremstyle{definition}

\theoremstyle{remark}
\newtheorem{remark}[theorem]{Remark}

\title{Constancy of an Infinite Cyclotomic Product via Ramanujan Sums}

\author{Hartosh Singh Bal}
\date{}

\begin{document}
\maketitle

\begin{abstract}
\noindent
We show that the infinite product defined by 
\[
P(z) = -\prod_{n=1}^{\infty} (\Phi_n(z))^{-1/n},
\]
where \( \Phi_n(z) \) is the \( n \)-th cyclotomic polynomial, is constant inside the unit disk. The proof translates a result of Ramanujan on Ramanujan sums, equivalent to the prime number theorem, to the setting of infinite products. We also show that similar identities proved by Ramanujan lead to additional results on infinite cyclotomic products.
\end{abstract}

\firstpagefoot{
\textbf{Author address:} The Caravan, Jhandewalan Extn., New Delhi 110001, India\\
\texttt{hartoshbal@gmail.com}\\[0.25em]

}

\section{Introduction}

The study of cyclotomic integers and their associated products plays an 
important role in number theory, with connections to modular forms, 
transcendence questions, and classical arithmetic functions. In this paper 
we revisit one such product, showing that its logarithmic derivative has 
a surprisingly simple description in terms of Ramanujan sums. Our main 
result is stated in Theorem~\ref{thm:main} below, which provides the 
starting point for the rest of the paper.

\begin{theorem}\label{thm:main}
The function
\[
P(z) = - \prod_{n=1}^{\infty} (\Phi_n(z))^{-1/n},
\]
where \( \Phi_n(z) \) denotes the \( n \)-th cyclotomic polynomial, is identically constant, \( P(z) \equiv 1 \), inside the unit disk \( |z| < 1 \).
\end{theorem}

Our approach uses a classical result by Ramanujan on Ramanujan sums \cite{4}, which is equivalent to the prime number theorem. It carries forward the idea of using the logarithmic derivative to derive information on an associated infinite product, as in \cite{1}. An earlier version of this work used a formal interchange of infinite sums. The present version replaces this step with a truncation-and-limit argument, eliminating any convergence ambiguity. The results are unchanged.

In subsequent sections, we construct new infinite product identities involving cyclotomic polynomials by invoking other results from Ramanujan's paper. To our knowledge, results of this type for infinite cyclotomic products have not appeared previously. A related example is \cite{2}, where the products are of the form $\prod_{k \ge 0} \Phi_\ell(z^{p^k})$ involving a single cyclotomic factor.

\section{Proof of Theorem 1.1}

Before we begin the proof, we recall that the \( n \)-th cyclotomic polynomial is defined as 
\[
\Phi_n(z)= \prod_{\substack{1 \leq k \leq n \\ \gcd(k,n)=1}} (z - e^{2\pi i k/n}),
\]
and Ramanujan sums, denoted by \( c_n(m) \), are given by
\[
c_n(m) = \sum_{\substack{1 \leq k \leq n \\ \gcd(k,n)=1}} e^{2\pi i k m/n}.
\]
These sums were first studied in detail by Ramanujan \cite{4}, who used them to express arithmetic functions in terms of trigonometric sums.

For convenience, define
\[
\hat{\Phi}_n(z) \;:=\; \prod_{\substack{1\le k\le n\\ (k,n)=1}} \bigl(1 - z\,\zeta_n^{\,k}\bigr),
\]
so that $\hat{\Phi}_1(z)=1-z$ and $\hat{\Phi}_n(z)=\Phi_n(z)$ for all $n\ge 2$.

\begin{proof}[Proof of Theorem~\ref{thm:main}]
We write the primitive \( n \)-th roots of unity \( e^{2\pi i k/n} \) as \( \zeta_{n_k} \); hence
\[
c_n(m) = \sum_{\substack{1 \leq k \leq n \\ \gcd(k,n)=1}} \zeta_{n_k}^m.
\]

We begin from the definition
\[
\Phi_n(z) \;=\; \prod_{\substack{1 \leq k \leq n \\ \gcd(k,n)=1}} \bigl(z - \zeta_n^{\,k}\bigr),
\]
where $\zeta_n = e^{2\pi i/n}$.  For each factor we write
\[
z - \zeta_n^{\,k} \;=\; \zeta_n^{\,k} \Bigl( \tfrac{z}{\zeta_n^{\,k}} - 1 \Bigr),
\]
so that
\[
\Phi_n(z) \;=\;
\Biggl(\prod_{\substack{(k,n)=1}} \zeta_n^{\,k}\Biggr)
\prod_{\substack{(k,n)=1}} \Bigl(\tfrac{z}{\zeta_n^{\,k}} - 1\Bigr).
\]

For $n>2$, the number of primitive residues $\varphi(n)$ is even and the primitive roots occur in inverse pairs $\xi,\xi^{-1}$, so their product is $1$.  For $n=2$, the only primitive root is $-1$, giving a product of $-1$, and for $n=1$ the product is $1$.  Pulling out the minus signs from each term $\tfrac{z}{\zeta_n^{\,k}} - 1 = -(1 - z \zeta_n^{-k})$ contributes a factor $(-1)^{\varphi(n)}$, which is $+1$ for $n>2$ and cancels the prefactor for $n=2$.  Thus, for all $n \geq 2$ we obtain
\[
\Phi_n(z) \;=\; \prod_{\substack{(k,n)=1}} \bigl(1 - z \zeta_n^{-k}\bigr).
\] Since the map $k \mapsto k^{-1}$ permutes the primitive residue classes, the set
$\{\zeta_n^{-k}\}$ is the same as $\{\zeta_n^{\,k}\}$, and therefore
\[
\Phi_n(z) \;=\; \prod_{\substack{(k,n)=1}} \bigl(1 - z \zeta_n^{\,k}\bigr),
\qquad n \geq 2.
\]

Finally, for the two base cases we record explicitly:
\[
\Phi_2(z) = z+1 = 1 - (-1)z, \qquad 
\Phi_1(z) = z-1 = -(1-z).
\]
In the case $n=1$ the product form gives $1-z$, which differs from $\Phi_1(z)$ by an overall sign.

Consider the truncated logarithms
\[
\Log P_N(z)\;:=\;-\sum_{n\le N}\frac{1}{n}\,\Log\hat\Phi_n(z),
\qquad (N\ge1),
\]
where $\Log\hat\Phi_n$ denotes the holomorphic branch on $|z|<1$ determined by
$\Log\hat\Phi_n(0)=0$ (so $\hat\Phi_n(0)=1$).  Thus $P_N(z):=\exp(\Log P_N(z))$
is well-defined and analytic on the open unit disk.

For $|z|<1$ we expand, for each fixed $n$,
\[
\Log\hat\Phi_n(z)=\sum_{\substack{1\le k\le n\\(k,n)=1}}\Log(1-z\zeta_{n_k})
=-\sum_{m\ge1}\frac{z^m}{m}\sum_{\substack{1\le k\le n\\(k,n)=1}}\zeta_{n_k}^{\,m}
=-\sum_{m\ge1}\frac{c_n(m)}{m}\,z^m,
\]
using $\Log(1-w)=-\sum_{m\ge1}w^m/m$ for $|w|<1$ and the definition of $c_n(m)$.
Substituting into the finite sum defining $\Log P_N$ and interchanging the
finite $n$--sum with the absolutely convergent $m$--series yields
\[
\Log P_N(z)
=\sum_{m\ge1}\frac{z^m}{m}\,S_N(m),
\qquad
S_N(m):=\sum_{n\le N}\frac{c_n(m)}{n}.
\]

Ramanujan's identity is to be read as the vanishing of these partial sums:
\begin{equation}\label{eq:Ram-vanish}
\lim_{N\to\infty} S_N(m)=\lim_{N\to\infty}\sum_{n\le N}\frac{c_n(m)}{n}=0
\qquad (m\ge1),
\end{equation}
a statement equivalent to the prime number theorem \cite{3,4}.
Moreover, writing $c_n(m)=\sum_{d\mid(n,m)} d\,\mu(n/d)$ gives
\[
S_N(m)=\sum_{d\mid m}\;\sum_{r\le N/d}\frac{\mu(r)}{r},
\]
and under \eqref{eq:Ram-vanish} the inner sums are uniformly bounded; hence
$|S_N(m)|\ll\tau(m)$ independently of $N$.  Since
$\sum_{m\ge1}\tau(m)|z|^m/m<\infty$ for $|z|<1$, dominated convergence in $m$
implies $\Log P_N(z)\to 0$ for every $|z|<1$. Therefore $P_N(z)\to 1$ on the
unit disk, and defining $P(z):=\lim_{N\to\infty}P_N(z)$ we obtain $P(z)\equiv 1$.

Conversely, assume that $P(z)\equiv 1$ on $|z|<1$, where $P$ is defined as the
limit of the truncations $P_N$ from above. For each $N$ we have
\[
\Log P_N(z)=\sum_{m\ge1}\frac{z^m}{m}\,S_N(m),
\qquad
S_N(m):=\sum_{n\le N}\frac{c_n(m)}{n},
\]
and the series converges absolutely on $|z|<1$. Differentiating term-by-term gives
\[
z\frac{P_N'(z)}{P_N(z)}=z\frac{d}{dz}\Log P_N(z)=\sum_{m\ge1} S_N(m)\,z^m.
\]
Since $P_N\to P\equiv 1$ uniformly on compact subsets of $|z|<1$, we have
$zP_N'(z)/P_N(z)\to 0$ uniformly on compact subsets as well, hence for each fixed $m$
the coefficients satisfy $S_N(m)\to 0$ as $N\to\infty$. Equivalently,
\[
\lim_{N\to\infty}\sum_{n\le N}\frac{c_n(m)}{n}=0\qquad(m\ge1),
\]
which is Ramanujan's identity in its correct (partial-sum/Abel) sense.

\end{proof}

\begin{remark}[Partial products and boundary behavior]
Although the analytic function $P(z)$ collapses to the constant~$1$ on $|z|<1$,
the truncations
\[
P_N(z):=\exp\!\Bigl(-\sum_{n\le N}\frac{1}{n}\Log\hat\Phi_n(z)\Bigr)
\qquad (|z|<1)
\]
form a natural approximation scheme and converge uniformly to~$1$ on compacta.
On the unit circle, the individual factors $\hat\Phi_n$ vanish at roots of unity
of order $n$, so the boundary profiles of $P_N$ develop increasingly dense
``spikes'' at roots of unity of order at most $N$. In this sense $\{P_N\}$ encodes
a boundary concentration phenomenon concentrated on rational angles, even though
the interior limit is analytically trivial.
\end{remark}

\begin{theorem}
For any real parameter $s>1$,
\[
\prod_{i \ge 1} \Bigl(1 - z^i\Bigr)^{-1/i^{s}} 
= \prod_{i \ge 1} \hat{\Phi}_i(z)^{-\zeta(s)/i^{s}}.
\]
\end{theorem}

\begin{proof}
Let $s>1$ and $|z|<1$. We interpret $\hat\Phi_i(z)^a$ as
$\exp(a\,\Log\hat\Phi_i(z))$, where $\Log\hat\Phi_i$ is the holomorphic branch on
$|z|<1$ with $\Log\hat\Phi_i(0)=0$.

A classical Ramanujan expansion asserts that for $s>1$,
\[
\sigma_{1-s}(n)=\zeta(s)\sum_{i\ge1}\frac{c_i(n)}{i^{s}},
\]
where $\sigma_\alpha(n):=\sum_{d\mid n} d^\alpha$.
Multiplying by $z^n/n$ and summing over $n\ge1$ gives
\[
\sum_{n\ge1}\frac{\sigma_{1-s}(n)}{n}z^n
=\zeta(s)\sum_{n\ge1}\sum_{i\ge1}\frac{c_i(n)}{i^{s}}\frac{z^n}{n}.
\]
All series are absolutely convergent for $|z|<1$ and $s>1$, so we may interchange sums.

On the left,
\[
\sum_{n\ge1}\frac{\sigma_{1-s}(n)}{n}z^n
= -\sum_{i\ge1}\frac{1}{i^{s}}\log(1-z^i)
=\log\!\Bigl(\prod_{i\ge1}(1-z^i)^{-1/i^{s}}\Bigr).
\]
On the right, using $\Log\hat\Phi_i(z)=-\sum_{n\ge1}c_i(n)z^n/n$ for $|z|<1$,
\[
\zeta(s)\sum_{i\ge1}\frac{1}{i^{s}}\sum_{n\ge1}\frac{c_i(n)}{n}z^n
= -\sum_{i\ge1}\frac{\zeta(s)}{i^{s}}\Log\hat\Phi_i(z)
=\log\!\Bigl(\prod_{i\ge1}\hat\Phi_i(z)^{-\zeta(s)/i^{s}}\Bigr).
\]
Exponentiating yields the claimed identity.
\end{proof}

\begin{remark}[Branch convention]
Throughout, for each $i\ge1$ we work on the unit disk $|z|<1$, where $\hat\Phi_i(z)\neq0$
and $\hat\Phi_i(0)=1$. We therefore fix the unique holomorphic branch $\Log\hat\Phi_i$
on $|z|<1$ normalized by $\Log\hat\Phi_i(0)=0$, and we interpret complex powers via
\[
\hat\Phi_i(z)^{a}:=\exp\!\bigl(a\,\Log\hat\Phi_i(z)\bigr),\qquad a\in\mathbb C.
\]
Similarly, $\Log(1-z^i)$ denotes the principal holomorphic branch on $|z|<1$ with
$\Log(1)=0$.
With this convention, all logarithmic identities below hold as identities of holomorphic
functions on $|z|<1$, and exponentiation yields the corresponding product identities.
In particular, the displayed ratio formula below is to be understood using these same holomorphic
branches of $\Log(1-z^i)$ and $\Log\hat\Phi_i(z)$ on $|z|<1$.
\end{remark}

\begin{remark}
Rearranging the argument above also yields, for $|z|<1$,
\[
\zeta(s+1)
= \frac{\sum_{i \ge 1}\frac{1}{i^{s+1}}\ln(1-z^i)}
       {\sum_{i \ge 2}\frac{1}{i^{s+1}}\ln(\Phi_i(z))+\ln(1-z)},
\]
an alternative expression that makes explicit the link between $\zeta(s+1)$
and cyclotomic factors.
\end{remark}

\section{Infinite Cyclotomic Products and the Theta Function}

We conclude with a classical identity of Ramanujan concerning the number
of representations of a positive integer $n$ as a sum of two squares.
Let $r_2(n)$ denote this number of representations. Ramanujan proved that
\[
\pi \sum_{i = 0}^{\infty} (-1)^i \frac{c_{2i+1}(n)}{2i+1} \; = \; r_2(n),
\qquad (n\ge1).
\]
A classical divisor-class formula gives
\[
r_2(n) \;=\; 4\bigl(d_1(n)-d_3(n)\bigr),
\]
where $d_1(n)$ and $d_3(n)$ count the divisors of $n$ congruent to $1$ and
$3$ modulo $4$, respectively.

\medskip
\noindent
\textbf{Branch convention.}
For $|z|<1$ we use the holomorphic branches of $\Log(1-z^m)$ and
$\Log\hat\Phi_n(z)$ normalized by $\Log(1)=0$ and $\Log\hat\Phi_n(0)=0$.
All complex powers are interpreted by
\[
A(z)^{\alpha}:=\exp\!\bigl(\alpha\,\Log A(z)\bigr),\qquad \alpha\in\mathbb{C}.
\]
(Here the hat only matters for the $n=1$ factor, since $\hat\Phi_n=\Phi_n$ for $n\ge2$.)

Applying the cyclotomic product method developed in the previous section,
we arrive at the infinite product identity
\[
\prod_{i = 0}^{\infty} \frac{(1 - z^{4i+3})^{4/(4i+3)}}
                           {(1 - z^{4i+1})^{4/(4i+1)}}
\; = \; \prod_{i = 0}^{\infty}
        \frac{\Phi_{4i+3}(z)^{\pi/(4i+3)}}
             {\hat\Phi_{4i+1}(z)^{\pi/(4i+1)}},
\qquad |z|<1.
\]

Taking the logarithmic derivative of this identity yields the power-series relation
\[
\pi\sum_{n=1}^\infty z^n
      \sum_{i=0}^\infty (-1)^i\frac{c_{2i+1}(n)}{2i+1}
\;=\; \sum_{n=1}^\infty r_2(n)z^n,
\qquad |z|<1,
\] in which the Ramanujan identity supplies the coefficient equality
\[
r_2(n)=\pi\sum_{i=0}^{\infty}(-1)^i\,\frac{c_{2i+1}(n)}{2i+1},
\qquad (n\ge1).
\]
.

The right-hand side is the $z$-expansion of $\theta_3(z)^2-1$, where
\[
\theta_3(z):=\sum_{m\in\mathbb{Z}} z^{m^2}.
\]
Equivalently, upon setting $z=q=e^{2\pi i\tau}$, the function $\theta_3(\tau)^2$
is a holomorphic modular form of weight~$1$ on~$\Gamma_0(4)$ \cite{5}. Hence
\[
\begin{aligned}
q\frac{d}{dq}\log(\cdots)
&= \sum_{n\ge1} r_2(n)\,q^n \\
&= \theta_3(\tau)^2-1,
\end{aligned}
\]
so the logarithmic derivative is (up to its constant term) a weight-$1$ modular form on
$\Gamma_0(4)$.

Taking logarithms of the infinite product identity gives
\begin{multline*}
\sum_{i \ge 0} \Biggl[\frac{4}{4i+3}\Log(1-z^{4i+3})
 - \frac{4}{4i+1}\Log(1-z^{4i+1})\Biggr] \\
= \sum_{i \ge 0} \Biggl[\frac{\pi}{4i+3}\Log\Phi_{4i+3}(z)
 - \frac{\pi}{4i+1}\Log\hat\Phi_{4i+1}(z)\Biggr],
\qquad |z|<1.
\end{multline*}

Finally, rearranging terms leads to the striking logarithmic identity
\[
\frac{\pi}{4} \;=\;
\frac{\displaystyle\sum_{i \ge 0} (-1)^{i+1}\,\frac{\Log(1-z^{2i+1})}{2i+1}}
{\displaystyle\sum_{i \ge 0} (-1)^{i+1}\,\frac{\Log\hat\Phi_{2i+1}(z)}{2i+1}},
\qquad |z|<1,
\]
where all logarithms are taken with the holomorphic branches fixed above.

\end{document}